\documentclass [12pt,oneside]{amsart}
\usepackage{color}
\usepackage{amssymb,amsmath,amsthm,amsfonts}
\usepackage{enumerate}
\usepackage{setspace,bbm}
\usepackage{empheq}
\usepackage[all]{xy}
\usepackage{verbatim}
\usepackage{longtable}
\usepackage[normalem]{ulem}
\usepackage{mathrsfs}
\usepackage{pinlabel}
\usepackage[percent]{overpic}
\usepackage{float}
\usepackage{caption}
\usepackage{subcaption}
\usepackage[bookmarks,colorlinks,breaklinks]{hyperref}  
\hypersetup{linkcolor=black,citecolor=black,filecolor=blue,urlcolor=black} 

\numberwithin{equation}{section}
\numberwithin{figure}{section}

\newtheorem{thm}{Theorem}[section]
\newtheorem{lem}[thm]{Lemma}
\newtheorem{prop}[thm]{Proposition}
\newtheorem*{definition*}{Definition}
\newtheorem{cor}[thm]{Corollary}
\newtheorem{conj}[thm]{Conjecture}
\newtheorem{question}[thm]{Question}
\theoremstyle{remark}
\newtheorem*{rmk}{Remark}
\theoremstyle{remark}

\newcommand{\C}{\mathbb{C}}
\newcommand{\fr}{\frac}

\newcommand{\CC}{\mathbb{C}}
\newcommand{\ZZ}{\mathbb{Z}}
     
\newcommand{\QQ}{\mathbb{Q}}

\usepackage{amssymb,fge}

\definecolor{myblue}{rgb}{0.6, 0.9, 1}

\makeatletter

\newcommand{\Rmnum}[1]{\expandafter\@slowromancap\romannumeral #1@}
\makeatother

\definecolor{myblue}{rgb}{0.6, 0.9, 1}
\definecolor{mygreen}{rgb}{0,0,1}
\definecolor{purple}{rgb}{0.6,0.2,1}
\definecolor{orange}{rgb}{0.8,0,0.2}

\newcommand{\Gal}{\operatorname{Gal}}

\begin{document}
\title[Polynomials with superattracting periodic points]{A Bogomolov property for the canonical height of maps with superattracting periodic points}
\author{Nicole R. Looper}
	\thanks{The author's research was supported by NSF grant DMS-1803021.}
		\begin{abstract} We prove that if $f$ is a polynomial over a number field $K$ with a finite superattracting periodic point and a non-archimedean place of bad reduction, then there is an $\epsilon>0$ such that only finitely many $P\in K^{\textup{ab}}$ have canonical height less than $\epsilon$ with respect to $f$. The key ingredient is the geometry of the filled Julia set at a place of bad reduction. We also prove a conditional uniform boundedness result for the $K$-rational preperiodic points of such polynomials, as well as a uniform lower bound on the canonical height of non-preperiodic points in $K$. We further prove unconditional analogues of these results in the function field setting. 
	\end{abstract}
\maketitle

\bigskip
\section{Introduction}

Given a morphism $f:\mathbb{P}^N\to\mathbb{P}^N$ of degree $d\ge2$ over a product formula field $K$, the \emph{canonical height} $\hat{h}_f(\alpha)$ of $\alpha\in\mathbb{P}^N(\bar{K})$ satisfies \[\hat{h}_f(\alpha)=h(\alpha)+O(1),\] where $h$ denotes the Weil height. As the canonical height is conjugation invariant, it provides a natural measure of arithmetic complexity from a dynamical point of view.

This article treats two kinds of problems concerning points of small canonical height relative to polynomials. These are inspired by parallel conjectures and results in the Diophantine setting. The first cluster of problems concerns points of small canonical height lying in the maximal abelian extension $K^{\textup{ab}}/K$. In the setting of abelian varieties, prototypical results include the following. \begin{thm}\cite{Zarhin} Let $K$ be a number field, and let $A/K$ be an abelian variety $K$-isogenous to a product $A_1\times\cdots\times A_r$ of simple abelian varieties. Then $A(K^{\textup{ab}})_{\textup{tors}}$ is finite if and only if none of the $A_i$, $1\le i\le r$ are of CM-type over $K$.\end{thm}

\begin{thm}\cite{BakerSilverman} Let $K$ be a number field, let $A/K$ be an abelian variety, let $\mathcal{L}$ be a symmetric ample line bundle on $A/K$, and let $\hat{h}:A(\bar{K})\to\mathbb{R}$ be the canonical height function associated to $\mathcal{L}$. There is an $\epsilon>0$ such that for all $P\in A(K^{\textup{ab}})$, either $P$ is a torsion point, or $\hat{h}(P)>\epsilon$.\end{thm}

This leads to two natural questions in the dynamical setting.

\begin{question}\label{question:Bog1} For which morphisms $f:\mathbb{P}^N\to\mathbb{P}^N$ of degree at least $2$ over a number field $K$ are there only finitely many preperiodic $P\in\mathbb{P}^N(K^{\textup{ab}})$?\end{question}

\begin{question}\label{question:Bog2} For which morphisms $f:\mathbb{P}^N\to\mathbb{P}^N$ of degree at least $2$ over a number field $K$  does there exist an $\epsilon>0$ such that for all $P\in\mathbb{P}^N(K^{\textup{ab}})$, either $\hat{h}_f(P)=0$ or $\hat{h}_f(P)\ge\epsilon$? \end{question}

Our first theorem provides a partial answer to Questions \ref{question:Bog1} and \ref{question:Bog2}. We will say that a nonarchimedean place $v$ of $K$ is a place of \emph{bad reduction} if its $v$-adic Julia set is not a Type II point in the Berkovich analytification.

\begin{thm}\label{thm:Bogomolov} Let $K$ be a product formula field, and let $f\in K[z]$ be a polynomial. Suppose $f$ has a superattracting finite periodic point, and a nonarchimedean place of bad reduction. Then there is an $\epsilon>0$ such that there are only finitely many $P\in K^{\textup{ab}}$ with \[\hat{h}_f(P)<\epsilon.\] In particular, $f$ has only finitely many preperiodic points in $K^{\textup{ab}}$.\end{thm}

These maps have a very particular filled Julia set structure at the places of bad reduction, one that is key to proving each of the main results in this paper. The archimedean counterpart of this filled Julia set structure is illustrated in Figure \ref{figure1}. We remark that a polynomial satisfying these hypotheses necessarily has degree at least $3$.

In forming dynamical analogues of results in arithmetic geometry, one often finds that post-critically finite (PCF) maps play the role of CM abelian varieties. One might be tempted to conjecture, therefore, that any PCF map defined over a number field $K$ has infinitely many preperiodic points in $K^{\textup{ab}}$. However, this is false. In \cite{DvornicichZannier}, Dvornicich and Zannier show that when $K=\mathbb{Q}$, a degree $d$ polynomial $f\in K[z]$ has infinitely many preperiodic points in $K^{\textup{ab}}$ if and only if $f$ is conjugate over $\overline{\mathbb{Q}}$ to $T_d(\pm z)$ or $(\pm z)^d$, where $T_d$ is the $d$-th Chebyshev polynomial. On the other hand, the polynomial $f(z)=z^2-1$ is PCF but is not conjugate to any of these maps. 

In a different direction, one can ask for bounds on the number of `small points' that hold uniformly across families of maps. The most well-known conjecture in this vein is the Uniform Boundedness Conjecture of Morton and Silverman, a major open problem in arithmetic dynamics.

\begin{conj}\label{conj:UBC}\cite{MortonSilverman} Let $d\ge2$, $N\ge1$, and let $K$ be a number field. Let $f:\mathbb{P}^N\to\mathbb{P}^N$ be a morphism of degree $d$ defined over $K$. There is a constant $B=B(d,N,[K:\QQ])$ such that $f$ has at most $B$ preperiodic points in $\mathbb{P}^N(K)$.\end{conj}

For $d\ge3$ and $K$ a number field or a one-dimensional function field of characteristic zero, let $\mathfrak{S}_{d,m}(K)$ be the set of polynomials of degree $d$ in $K[z]$ having a superattracting periodic point $P\ne\infty$ of period $m$. We prove a conditional uniform boundedness result for maps in $\mathfrak{S}_{d,m}(K)$.

\begin{thm}\label{thm:UBCsuperattracting} Let $K$ be a number field or a one-dimensional function field of characteristic zero. Let $d\ge3$, $m\ge1$, and $f\in\mathfrak{S}_{d,m}(K)$. If $K$ is a number field, assume the $abc$-conjecture for $K$. If $K$ is a function field, assume $f$ is not isotrivial. There is a $B=B(d,m,K)$ such that $f$ has at most $B$ preperiodic points contained in $K$.\end{thm}

We also prove a theorem about points of small yet nonzero canonical height. The conclusion of this theorem can be viewed as a higher-dimensional, dynamical parallel to Lang's conjecture on the minimal canonical height of nontorsion points on elliptic curves over a number field. 

\begin{thm}\label{thm:mincanht}Let $K$ be a number field or a one-dimensional function field of characteristic zero. Let $d\ge3$, $m\ge1$, and $f(z)\in\mathfrak{S}_{d,m}(K)$. If $K$ is a number field, assume the $abc$-conjecture for $K$. There is a $\kappa=\kappa(d,m,K)>0$ such that \[\hat{h}_f(P)\ge\kappa\max\{1,h_{\textup{crit}}(f)\}\] for all non-preperiodic points $P\in K$.\end{thm}

\begin{rmk} It is possible for $f\in\mathfrak{S}_{d,m}(K)$ to satisfy $h_{\textup{crit}}(f)=0$ (i.e., to have $f$ be postcritically finite) and yet for $f$ to have a place of bad reduction. An example, due to Laura DeMarco, is the function $f(z)=-\frac{2}{9}z^3-z^2\in\QQ[z]$, which has $v=3$ as a place of bad reduction. Hence, the lower bound in Theorem \ref{thm:mincanht} must be $\kappa\max\{1,h_{\textup{crit}}(f)\}$ (as opposed to $\kappa h_{\textup{crit}}(f)$) to give the statement content in this case.\end{rmk}

A key ingredient in the proofs of Theorems \ref{thm:UBCsuperattracting} and \ref{thm:mincanht} is a global quantitative equidistribution result (Theorem \ref{thm:globaleq}) for points of small canonical height. A crucial feature of this theorem is that it holds uniformly across all polynomials $f\in\mathfrak{S}_{d,m}(K)$ of degree $d$. This result provides a way of measuring a certain kind of equidistribution of a set of points across all places simultaneously, and says that given any desired `quality of equidistribution' a sufficiently large set of sufficiently small points necessarily achieves this quality. We also leverage the special geometry of the Julia sets at places of bad reduction for these maps. It is ultimately this geometry that allows for the $abc$-conjecture, a Diophantine approximation statement on curves, to be put to fruitful use in proving dynamical uniformity results across families having dimension much larger than $1$. It should be noted that the subspace $\mathfrak{S}_{d,m}(K)$ of the moduli space $\mathcal{P}_d$ of degree $d$ polynomials over $K$ has codimension $1$, and hence is affine of dimension $d-2$ over $K$. By contrast, the Latt\`{e}s maps descended from multiplication-by-$m$ maps on elliptic curves form a one-dimensional family in the moduli space $\mathcal{M}_{m^2}$ of degree $m^2$ rational functions. Theorems \ref{thm:Bogomolov}, \ref{thm:UBCsuperattracting}, and \ref{thm:mincanht} thus address a larger class of maps than their existing counterparts in the setting of elliptic curves, in the sense that these maps form a much higher dimensional subspace of $\mathcal{M}_d$ when $d$ is a large square.

\indent\textbf{Acknowledgements.} It is a pleasure to thank Matt Baker, Laura DeMarco, and Holly Krieger for helpful conversations related to this project. I also thank Arnaud Plessis for useful comments on a draft of the paper.

\section{Notation and Background}\label{section:background}

\setlength{\tabcolsep}{15pt} 

\begin{longtable}{r p{10cm}} $K$ & unless otherwise stated, either a number field or a finite extension of the rational function field $k(t)$ over a characteristic $0$ field $k$ assumed without loss to be algebraically closed\\ $F$ & $[K:\QQ]$ is $K$ is a number field, and $k(t)$ if $K$ is a function field \\ $M_K$ & a complete set of inequivalent places of $K$, with absolute values $|\cdot|_v$ normalized to extend the standard absolute values on $\QQ$ if $K$ is a number field, or $k(t)$ if $K$ is a function field\\ $M_K^0$ & the set of nonarchimedean places of $K$ \\ $M_K^\infty$ & the set of archimedean places of $K$ \\ $\mathscr{S}_d$ & the set of places $v\in M_K$ such that $v$ lies above some prime integer less than or equal to $d$ \\ $k_{\mathfrak{p}}$ & the residue field associated to the finite prime $\mathfrak{p}$ of $K$ \\
	$N_\mathfrak{p}$ & $\frac{[k_\mathfrak{p}:f_\mathfrak{p}]}{[K:F]}$, where $k_{\mathfrak{p}}/f_\mathfrak{p}$ is the residual extension associated to the finite prime $\mathfrak{p}$ of $K$ \\ $r_v$ & $\frac{[K_v:F_v]}{[K:F]}$ \\ $v_\mathfrak{p}$ & the standard $\mathfrak{p}$-adic valuation on $K$
\end{longtable} 

\noindent If $K$ is a number field, let $\mathcal{O}_K$ denote the ring of integers. If $K$ is a function field, let $\mathcal{O}_K$ denote the integral closure of $k[t]$ in $K$.
If $K$ is a number field, $n\ge 2$ and $P=(z_1,\dots,z_n)\in\mathbb{P}^{n-1}(K)$ with $z_1,\dots,z_n\in K$, let \begin{equation*}\begin{split}h(P)=&\sum_{\textup{primes }\mathfrak{p} \textup{ of } \mathcal{O}_K} -\min\{v_{\mathfrak{p}}(z_1),\dots,v_{\mathfrak{p}}(z_n)\}N_{\mathfrak{p}}\\&+\dfrac{1}{[K:\QQ]} \sum_{\sigma:K\hookrightarrow\CC} \log\max\{|\sigma(z_1)|,\dots,|\sigma(z_n)|\},\end{split}\end{equation*} where we do not identify conjugate embeddings. (We choose to express the height in this form, which separates the nonarchimedean and archimedean contributions, for convenience in applying the $abc$-conjecture.) If $K$ is a function field, let \[h(P)=\sum_{\textup{primes }\mathfrak{p} \textup{ of } \mathcal{O}_K} -\min\{v_{\mathfrak{p}}(z_1),\dots,v_{\mathfrak{p}}(z_n)\}N_{\mathfrak{p}}.\] For any $P=(z_1,\dots,z_n)\in\mathbb{P}^{n-1}(K)$ with $z_1,\dots,z_n\in K^*$, we define \[I(P)=\{\textup{primes }\mathfrak{p} \textup{ of } \mathcal{O}_K\mid v_{\mathfrak{p}}(z_i)\ne v_{\mathfrak{p}}(z_j)\textup{ for some } 1\le i,j\le n\}\] and let \[\textup{rad}_K(P)=\sum_{\mathfrak{p}\in I(P)} N_{\mathfrak{p}},\] where the subscript $K$ emphasizes the dependence of this quantity on $K$. 

Theorems \ref{thm:UBCsuperattracting} and \ref{thm:mincanht} will make use of the $abc$-conjecture for a number field $K$. \begin{conj}[The $abc$-conjecture]{\label{conj:nconj}} Let $K$ be a number field, and let $[Z_1:Z_2:Z_3]$ be the standard homogeneous coordinates on $\mathbb{P}^2(K)$. For any $\epsilon>0$, there is a constant $C_{K,\epsilon}$ such that for all $P=(z_1,z_2,z_3)\in\mathbb{P}^2(K)$ with nonzero coordinates satisfying $z_1+z_2+z_3=0$, we have \[h(P)<(1+\epsilon)\textup{rad}_K(P)+C_{K,\epsilon}.\]\end{conj}

For $K/k(t)$ a function field of characteristic zero, the $abc$-conjecture was proved independently by Mason \cite{Mason} and Silverman \cite{Silverman:abc}.

\subsection{Nonarchimedean analysis and quantitative equidistribution} We now introduce concepts from non-archimedean dynamics that will be used in proving our main results. Let $(\CC_v,|\cdot|_v)$ be a complete algebraically closed nontrivially valued field. For $f(z)\in\CC_v[z]$ of degree $d\ge2$ and $z\in\C_v$, let \[\hat{\lambda}_v(z)=\lim_{n\to\infty} \fr{1}{d^n}\log^+|f^n(z)|_v\] be the standard $v$-adic escape-rate function. (See \cite[\S3.4, 3.5]{Silverman3} for a proof that the limit defining $\hat{\lambda}_v(z)$ exists.) Note that $\hat{\lambda}_v(z)$ obeys the transformation rule \begin{equation*}\hat{\lambda}_v(f(z))=d\hat{\lambda}_v(z)\end{equation*} for all $z\in \CC_v$. 

Suppose $|\cdot|_v$ is nonarchimedean. For every $g\in\mathbb{R}_{>0}$, Proposition \ref{prop:nonarchdiskpreimage} implies that the set of points $z\in\CC_v$ such that $\hat{\lambda}_v(z)\le g$ is a finite union of $r_g$ disjoint closed disks. Call $g\in\mathbb{R}_{>0}$ the ($v$-adic) \emph{splitting potential} of $f$ if $r_g=1$ and $r_{g'}>1$ for any $g'<g$. If $g$ is the splitting potential of $f$, the (logarithmic $v$-adic) \emph{splitting radius} is the log of the radius of the disk $\{z\in\CC_v:\hat{\lambda}_v(z)\le g\}$. 

A place $v\in M_K^0$ is said to be a place of \emph{bad reduction}, or alternatively, a \emph{bad place} for $f\in K[z]$ if it has a splitting potential. If $R_f$ is the set of finite critical points of $f$, the \emph{$v$-adic critical height} is \[\lambda_{\textup{crit},v}(f):=\max_{a\in R_f}\{\hat{\lambda}_v(a)\}.\] For $f(z)\in K[z]$, let \begin{equation}\label{eqn:hcrit}h_{\textup{crit}}(f)=\sum_{v\in M_K}r_v\lambda_{\textup{crit},v}(f).\end{equation} Note that the critical height is usually defined as \[\sum_{P\in\mathbb{P}^1(\bar{K})}(e_f(P)-1)\hat{h}_f(P),\] where $e_f(P)$ denotes the ramification index of $f$ at $P$. The latter definition is easily seen to be comparable to (\ref{eqn:hcrit}), but is less natural from a dynamical point of view. When $v\in M_K^0\setminus\mathscr{S}_d$ is a place of bad reduction for a monic polynomial $f(z)\in K[z]$ of degree $d$, the splitting radius is equal to the splitting potential, which in turn equals the $v$-adic critical height \cite[Lemmas 2.1 and 2.2]{Ingram:PCF}. 

We will be using a measure of the size of a set of bad places for a given polynomial $f(z)\in K[z]$.

\begin{definition*} For $0<\delta<1$, and $\Sigma\subseteq M_K^0$, a \emph{$\delta$-slice of places} $v\in\Sigma$ is a set $S$ of bad places $v\in\Sigma$ of $f$ such that \[\sum_{v\in S}r_v\lambda_{\textup{crit},v}(f)\ge\delta\sum_{v\in\Sigma}r_v\lambda_{\textup{crit},v}(f).\]\end{definition*}

In order to carry out our analysis in the nonarchimedean setting, we work in Berkovich space. The Berkovich affine line $\textbf{A}_v^1$ over $\CC_v$ is the collection of multiplicative seminorms on $\CC_v[T]$ which extend the norm $|\cdot|_v$ on $\CC_v$. Let $[\phi]_{D(a,r)}=\sup_{z\in D(a,r)}|\phi(z)|_v$ denote the sup-norm on the disk $D(a,r)$. The Berkovich classification theorem (see \cite[Theorem 2.2]{BakerRumely}) states that each seminorm $[\cdot]_x$ corresponds to an equivalence class of sequences of nested closed disks $\{D(a_i,r_i)\}$ in $\CC_v$, by the identification \[[\phi]_x:=\lim_{i\to\infty}[\phi]_{D(a_i,r_i)}.\] We define the action of polynomial $f\in\CC_v[T]$ on a point $x\in\textbf{A}_{v}^1$ by $[\phi]_{f(x)}=[\phi\circ f]_x$ for $\phi\in\CC_v[T]$. 

For $a\in\CC_v$, we define open and closed Berkovich disks of radius $r>0;$ \[\mathcal{B}(a,r)^-=\{x\in\textbf{A}_v^1:[T-a]_x<r\},\] \[\mathcal{B}(a,r)=\{x\in\textbf{A}_v^1:[T-a]_x\le r\}.\] An \emph{annulus} in $\mathbf{A}_v^1$ is a set of the form $\mathcal{B}\setminus\mathcal{B}'$ where $\mathcal{B}'\subsetneq\mathcal{B}$ are disks (either open or closed) of positive radius. A basis for the open sets of $\textbf{A}_v^1$ is given by sets of the form $\mathcal{B}(a,r)^-$ and $\mathcal{B}(a,r)^-\setminus\cup_{i=1}^N\mathcal{B}(a_i,r_i)$, where $a,a_i\in \CC_v$ and $r,r_i>0$. The Berkovich $v$-adic filled Julia set of $f(z)\in K[z]$ is defined as \[\mathcal{K}_v=\bigcup_{M>0}\{x\in\textbf{A}_v^1:[f^n(z)]_x\le M\text{ for all } n\ge0\}.\] A place $v\in M_K^0$ is a place of bad reduction for $f\in K[z]$ if and only if $\mathcal{K}_v$ is not a disk in $\mathbf{A}_v^1$.

For a finite set $T\subseteq\CC_v$, and $\delta_v(x,y)$ the Hsia kernel relative to $\infty$ (cf. \cite[Section 4.1]{BakerRumely}), let \[d_v(T)=\prod_{z_i\ne z_j}\delta_v(z_i,z_j)^{1/(n(n-1))}.\] For disks $\mathcal{B},\mathcal{B}'$, we write $\delta_v(\mathcal{B},\mathcal{B}')$ to mean $\text{sup}_{x\in\mathcal{B},y\in\mathcal{B}'}\delta_v(x,y)$. For a compact set $E\subseteq\mathbf{A}_v^1$, let \[\gamma_v(E)=\lim_{n\to\infty}\sup\left\{\prod_{i\ne j}\delta_v(z_i,z_j)^{1/(n(n-1))}:z_1,\dots,z_n\in E\right\}.\]

\begin{prop}\cite[Proposition 7.33]{Benedetto}{\label{prop:nonarchdiskpreimage}} Let $\mathcal{B}$ be a disk in $\mathbf{A}_v^1$, and let $\phi\in\CC_v[z]$ be a polynomial of degree $d\ge 1$. Then $\phi^{-1}(\mathcal{B})$ is a union $\mathcal{B}_1\cup\cdots\cup\mathcal{B}_m$ of disjoint disks, where $1\le m\le d$. Moreover, for each $i=1,\dots,m$, there is an integer $1\le d_i\le d$ such that every point in $\mathcal{B}$ has exactly $d_i$ pre-images in $\mathcal{B}_i$, counting multiplicity, and $d_1+\cdots+d_m=d$.
\end{prop} The $\mathcal{B}_i$ in Proposition \ref{prop:nonarchdiskpreimage} are referred to as the \emph{disk components} of $\phi^{-1}(\mathcal{B})$. 

One tool we will use is a proposition controlling how the modulus of an annulus transforms under a covering map. Given an annulus $\mathcal{A}:=\mathcal{B}\setminus\mathcal{B}'$, the modulus of $\mathcal{A}$ is by definition $\log\delta_v(\mathcal{B},\mathcal{B})-\log\delta_v(\mathcal{B}',\mathcal{B}')$.

\begin{prop}\cite[Corollary 2.6]{Baker}{\label{prop:nonarchmoduli}} If $\phi\in\C_v(z)$ is a degree $k$ covering map $\phi:\mathcal{A}_1\to \mathcal{A}_2$ of annuli in $\mathbf{A}_v^1$, then \[\textup{mod}(\mathcal{A}_1)=\fr{1}{k}\textup{mod}(\mathcal{A}_2).\] \end{prop} 

For $f(z)\in\CC_v[z]$ be a monic polynomial having bad reduction, let $g_v$ denote the splitting radius of $f$. Let $\mathcal{E}$ be the Berkovich closure of the disk $\{z\in\CC_v:\log|z|_v\le g_v\}$. This is the smallest Berkovich disk containing $\mathcal{K}_v$. For all $m\ge0$, let \[\mathcal{E}_m=f^{-m}(\mathcal{E})=\bigcup_{i=1}^{d^m}\mathcal{B}_{m,i},\] where $\mathcal{B}_{m,1},\dots,\mathcal{B}_{m,d^m}$ are the disk components of $\mathcal{E}_m$, listed with multiplicity. 

A \emph{wing decomposition} of $\mathcal{E}_1$ is a partition of $\mathcal{E}_1$ into two nonempty disjoint sets (wings) $A$ and $B$ with the following properties:
\begin{itemize}\item $A$ and $B$ are unions of disk components of $\mathcal{E}_1$ \item For any disk components $\mathcal{B}_{1,i}$ of $A$ and $\mathcal{B}_{1,j}$ of $B$, we have $\log\delta_v(\mathcal{B}_{1,j},\mathcal{B}_{1,l})=g_v$. \end{itemize} A wing decomposition is not unique in general.

We introduce a definition that will serve as a way to measure the level of equidistribution of sets of points of $K$ at places of bad reduction of $f\in\mathfrak{S}_{d,m}(K)$. 

\begin{definition*} Let $\epsilon>0$, let $m_0\ge1$, and let $v$ be a place of bad reduction for $f\in\mathfrak{S}_{d,m}(K)$ with finite superattracting point $p_0$ of period $m$. Let $\mathcal{E}$ be the smallest disk containing $\mathcal{K}_v$. Let $\mathcal{A}(m_0)$ be the half-open annulus $\mathcal{B}\setminus\mathcal{B}'$, where $\mathcal{B}$ is the disk component of $f^{-m_0\cdot m}(\mathcal{E})$ containing $p_0$, and $\mathcal{B}'$ is the disk component of $f^{-(m_0+1)m}(\mathcal{E})$ containing $p_0$. Let $\mu_v$ be the equilibrium measure on $\mathcal{K}_v$. We say that a finite set $T\subseteq\textbf{A}_v^1$ is $\epsilon$-equidistributed at level $m_0$ if \[(1-\epsilon)\mu_v(\mathcal{A}(m_0)\cap\mathcal{K}_v)<\frac{|T\cap\mathcal{A}(m_0)|}{|T|}<(1+\epsilon)\mu_v(\mathcal{A}(m_0)\cap\mathcal{K}_v),\] and if for any wing decomposition of $\mathcal{E}_1$ as above, we have \[|T\cap A|>\left(\frac{1-\epsilon}{d}\right)|T|.\] \end{definition*}

An analogue of the proof of \cite[Theorem 3.1]{Looper:UBCpolys} leads to the following theorem, which is key to the proofs of Theorems \ref{thm:UBCsuperattracting} and \ref{thm:mincanht}. 

\begin{thm}[Global quantitative equidistribution]\label{thm:globaleq}Let $d\ge3$, and let $f\in\mathfrak{S}_{d,m}(K)$ have finite superattracting point $p_0$ of period $m$. Let $\epsilon>0$, $m_0\ge1$, and $0<\delta<1$. Let $T\subseteq K$ be a finite set. There are constants $N$ and $\kappa>0$, depending only on $d$, $[K:F]$, $\delta$, $m$, $m_0$, and $\epsilon$, such that if $|T|\ge N$ and \[\frac{1}{|T|}\sum_{P_i\in T}\hat{h}_f(P_i)\le\kappa h_{\textup{crit}}(f),\] then $T$ is $\epsilon$-equidistributed at level $m_0$ for a $\delta$-slice of bad places $v\in M_K^0\setminus\mathscr{S}_d$.\end{thm}

\section{Filled Julia sets for maps in $\mathfrak{S}_{d,m}$}

In this section, we describe key features of filled Julia sets for polynomials having a finite superattracting periodic point, and a nonarchimedean place of bad reduction. At the place of bad reduction, these filled Julia sets have a `splash' emanating from the immediate basin of attraction of the superattracting periodic point. An archimedean version of this phenomenon can be seen in Figure \ref{figure1}. The results in this section apply when $K$ is any product formula field.

\begin{figure}[H]
	\centering
	\begin{subfigure}[t]{0.52\textwidth}
		\includegraphics[scale=0.45]{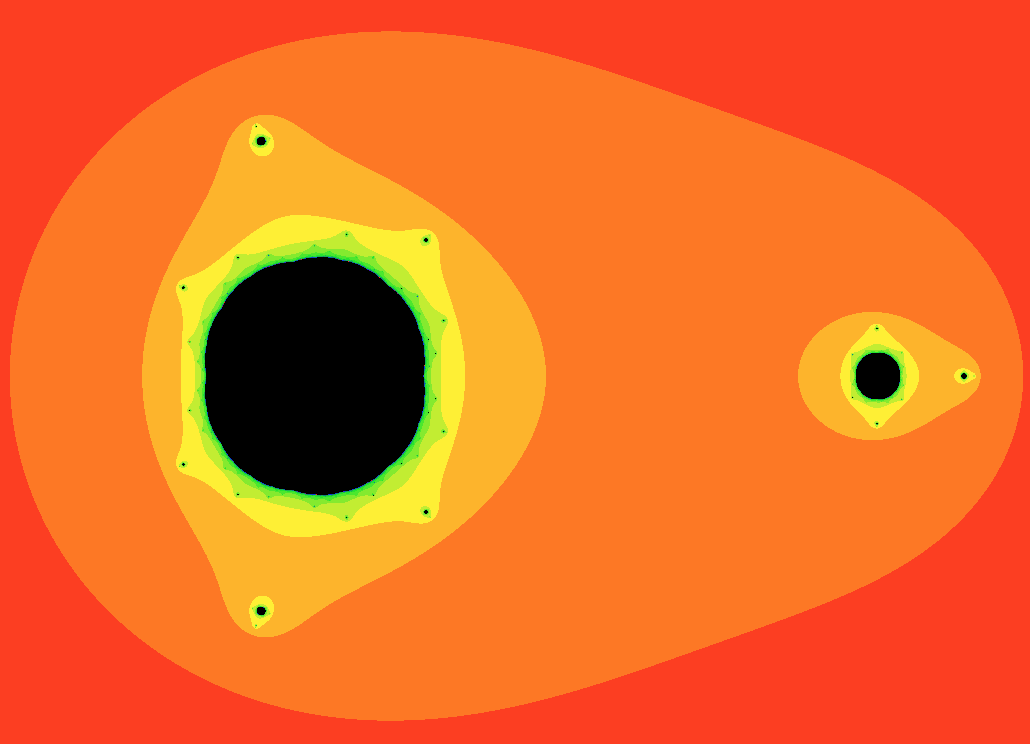}
		\caption{The equipotential curves of the map \\ $f(z)=\fr{1}{5}z^3-z^2$, at $v=\infty$.}
		\label{figure:splash}
	\end{subfigure}%
	~
	\begin{subfigure}[t]{0.52\textwidth}
		\includegraphics[scale=0.438]{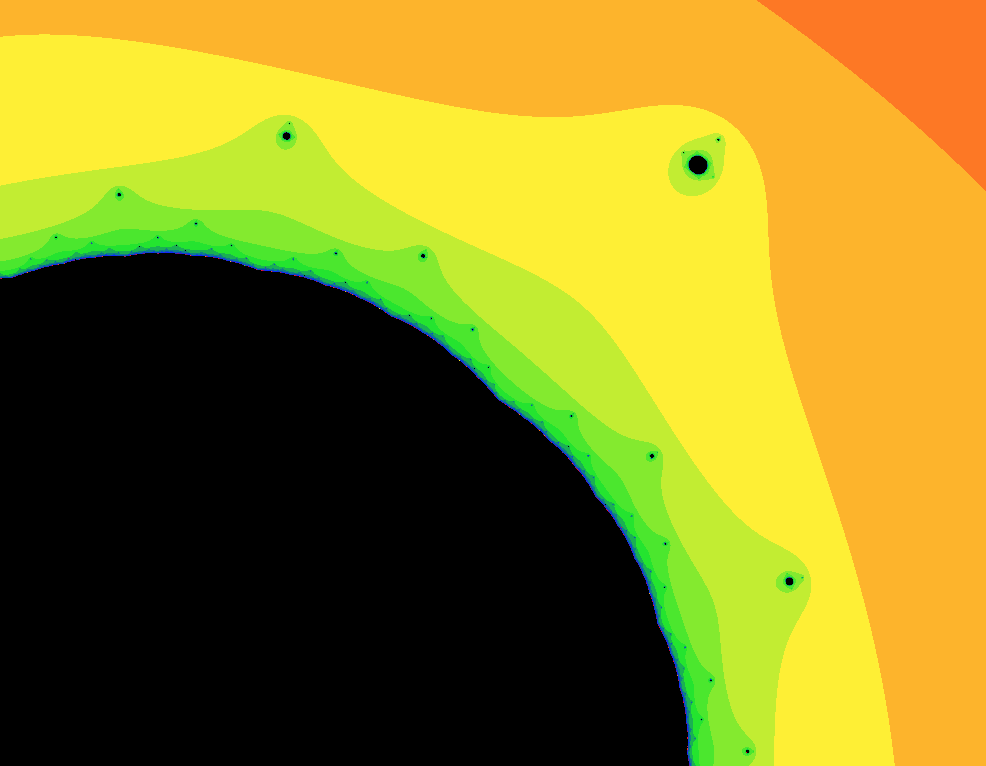}
		\caption{A zoomed in view of the immediate basin of $0$ in Figure \ref{figure:splash}.}
	\end{subfigure}
	\caption{\label{figure1}}
\end{figure}

We simplify the presentation by referring to maps that have a finite superattracting \emph{fixed} point, which eliminates the need to treat superattractors of different periods. 

\begin{prop}\label{prop:shape}Let $f(z)\in K[z]$ be a polynomial of degree $d\ge3$ with a finite superattracting fixed point $p_0$, and a nonarchimedean place $v$ of bad reduction. Let $r>0$ be the radius of the largest disk about $p_0$ contained in $\mathcal{K}_v$. There is a sequence of Berkovich disks $\mathcal{B}_i\subseteq\mathcal{K}_v$ such that $\delta_v(z_i,p_0)=r_i>r$ for all $z_i\in\mathcal{B}_i$, and $r_i\to r$ as $i\to\infty$.\end{prop}

\begin{proof} Write \[\mathcal{E}_1=\bigcup_{i=1}^r\mathcal{B}_{1,i},\] where the $\mathcal{B}_{1,i}$ are disjoint and $\mathcal{B}_{1,1}$ is the disk component of $\mathcal{E}_1$ containing $p_0$. Since $\mathcal{B}_{1,2}$ maps over $\mathcal{E}_0$, it contains an $f$-preimage $\mathcal{D}_1$ of $\mathcal{B}_{1,1}$. Similarly, $\mathcal{B}_{1,1}$ contains an $f$-preimage $\mathcal{D}_2$ of $\mathcal{D}_1$. If $\mathcal{B}_{2,1}$ is a disk component of $\mathcal{E}_2$ containing $p_0$ and mapping over $\mathcal{B}_{1,1}$, then $\mathcal{B}_{2,1}$ contains an $f$-preimage $\mathcal{D}_3$ of $\mathcal{D}_2$. If $\mathcal{B}_{3,1}$ is the disk component of $\mathcal{E}_3$ containing $p_0$, then $\mathcal{B}_{3,1}$ contains an $f$-preimage $\mathcal{D}_4$ of $\mathcal{D}_3$. Continuing in this manner, we produce a sequence of disks $\mathcal{D}_i$ whose distances $r_i$ from $p_0$ approach $r$ as $i\to\infty$.\end{proof}

From this we deduce two corollaries. In what follows, we fix an embedding of $\bar{K}$ into $\bar{K}_v$, and hence into $\mathbf{A}_v^1$.

\begin{cor}\label{cor:shape} Let $f(z)\in K[z]$ be a polynomial of degree $d\ge3$ with a finite superattracting periodic point, and a nonarchimedean place $v$ of bad reduction. Let $\{P_i\}$ be a sequence of points in $\bar{K}$ with $\hat{h}_f(P_i)\to 0$, and let $L^i$ denote the Galois closure of $K(P_i)/K$. For each $i$, let $\mathfrak{Q}_i$ be a prime of $\mathcal{O}_{L^i}$ lying above $v$. Then there is a subsequence $\{P_{i_j}\}$ with $|\Gal(L_{\mathfrak{Q}_{i_j}}^{i_j}/K_v)|\to\infty$ as $i_j\to\infty$.\end{cor}

\begin{proof} For each $n\ge1$, let $S_n=\bigcup_{i=2^{n-1}}^{2^n}P_i$, and let $\delta_n$ be the discrete probability measure supported equally on all points of $\textup{Gal}(\bar{K}/K)\cdot S_n$. By \cite[Theorem 10.24]{BakerRumely}, the sequence of measures $\delta_n$ converges weakly to the equilibrium measure on $\mathcal{K}_v$, which is supported on each of the $\mathcal{B}_i$ appearing in Proposition \ref{prop:shape}. It follows that there is a subsequence $\{P_{i_j}\}$ of the $P_i$ such that the ramification index of any prime $\mathfrak{Q}_{i_j}$ of $\mathcal{O}_{L^{i_j}}$ lying over $v$ tends to infinity as $i\to\infty$. \end{proof}

\begin{rmk} If $K$ has the Northcott property, then we need not pass to a subsequence in Corollaries \ref{cor:shape} and \ref{cor:basinnbhd}, since in that case $|\Gal(L^i/K)|\to\infty$ as $i\to\infty$.\end{rmk}

\begin{cor}\label{cor:basinnbhd}Let $f$, $v$ and $p_0$ be as above, and let $\mathcal{B}$ be the largest disk about $p_0$ contained in $\mathcal{K}_v$. Let $\mathcal{B}'$ be a Berkovich disk properly containing $\mathcal{B}$, let $\{P_i\}$ be a sequence of points in $\bar{K}$ with $\hat{h}_f(P_i)\to 0$ as $i\to\infty$, and let $\{P_{i_j}\}$ be the corresponding subsequence in the conclusion of Corollary \ref{cor:shape}. Let $L^i$ be as above, and let $s_i$ be the number of $\Gal(L^i/K)$ conjugates of $P_i$. Then there is an $\epsilon=\epsilon(f,\mathcal{B}')>0$ such that for all $i_j\gg1$, at least $\epsilon s_{i_j}$ of the $\Gal(L^{i_j}/K)$ conjugates of $P_{i_j}$ are contained in $\mathcal{B}'$.
\end{cor}

\begin{proof}The proof follows from the same reasoning as Corollary \ref{cor:shape}, relying on \cite[Theorem 10.24]{BakerRumely}.\end{proof}

For $f\in K[z]$ a monic polynomial of degree $d\ge3$ such that $0$ is a superattracting fixed point, $v\in M_K^0\setminus\mathscr{S}_d$ a place of bad reduction, and $g_v$ the splitting radius of $f$ at $v$, let $\mathcal{B}_1$ be the Berkovich disk component of $f^{-1}(\mathcal{B}(0,e^{g_v}))$ containing $0$. For all $i\ge2$, let $\mathcal{B}_i$ be the unique disk component of $f^{-i+1}(\mathcal{B}_1)$ containing $0$, so that $\mathcal{B}_{i+1}\subseteq\mathcal{B}_i$. Write $r(\mathcal{B}_i)=\log\delta_v(\mathcal{B}_i,\mathcal{B}_i)$.

From here until the proof of Theorem \ref{thm:Bogomolov}, let $K$ be a number field or a one-dimensional function field of characteristic zero. We now introduce a proposition that will be needed in the proofs of Theorems \ref{thm:UBCsuperattracting} and \ref{thm:mincanht}. This result says that $\log_v\delta_v(\mathcal{B}_i,\mathcal{B}_i)$ has height tending to infinity as $i\to\infty$ --- something that is apparent enough from Proposition \ref{prop:shape}, but the key refinement is that the growth of this height can be bounded from below independently of the map and the place of bad reduction. This will allow us to conclude from the  hypothetical $K$-rationality of a ``well-equidistributed" large set of points of small canonical height that the multiplicative height of the map at $v$ is a large power of a uniformizer, a property that then tends to be inherited by most differences of preperiodic points (or more generally, points of small canonical height). Global quantitative equidistribution forces this to happen at a large proportion of the bad places $v$, making the differences of preperiodic points tend to be globally `powerful.' It is the `powerful' nature of these differences that constitutes the chief ingredient in the proofs of Theorems \ref{thm:UBCsuperattracting} and \ref{thm:mincanht}.

\begin{prop}\label{prop:uniformram}Let $f\in K[z]$ be a monic polynomial of degree $d\ge3$ with $0$ a superattracting fixed point, and let $v\in M_K^0\setminus\mathscr{S}_d$ be a place of bad reduction. If, for all $i\ge1$, $q_i$ is the least integer such that \[\textup{exp}(r(\mathcal{B}_i)q_i)\in|K_v^{\times}|,\] then given $k\in\mathbb{N}$, there is an $I=I(d,k)$ such that if $i\ge I$, then $q_i\ge k$.\end{prop}

\begin{proof}Write $\mathcal{A}_i=\mathcal{B}_i\setminus\overline{\mathcal{B}_{i+1}}$ for all $i\ge1$. If $\mathcal{A}_{-1}$ is the annulus given by $f(\mathcal{B}(0,e^{g_v})^-)\setminus\mathcal{B}(0,e^{g_v})$, then $\textup{mod}(\mathcal{A}_{-1})=(d-1)g_v$. (Here it is crucial that $g_v$ is also equal to the splitting potential, as $f$ is monic and $v\notin\mathscr{S}_d$.) Let $\mathcal{A}_0$ be the annulus $\mathcal{B}(0,e^{g_v})^-\setminus\mathcal{B}_1$. By Proposition \ref{prop:nonarchmoduli}, it follows that \[g_v=(d-1)g_v/(d-1)\le\textup{mod}(\mathcal{A}_0)\le(d-1)g_v/2\] (cf. \cite[proof of Proposition 4.3]{Looper:mincanht}). By induction, we obtain \begin{equation}\label{eqn:modbounds}\frac{g_v}{(d-1)^i}\le\textup{mod}(\mathcal{A}_i)\le\frac{(d-1)g_v}{2^{i+1}}.\end{equation} To finish, we claim that from the Newton polygon of $f$ at $v$ one sees that there is some $1\le l\le d-2$ such that \begin{equation}\label{eqn:outerrad}e^{g_vl}\in|K_v^{\times}|\end{equation} Indeed, the splitting radius can be read off the slope of the rightmost segment of the Newton polygon. Since $f$ is assumed to be monic with bad reduction at $v$ and lowest degree term having degree at least $2$, this slope is a negative rational number, whose denominator in lowest terms is at most $d-2$. Combining (\ref{eqn:modbounds}) and (\ref{eqn:outerrad}) completes the proof.\end{proof}

The next lemma can be viewed as a description of the \emph{moment} of the $v$-adic absolute values of preperiodic points at a place $v$ of bad reduction of a monic polynomial $f$ having $0$ as a preperiodic point. Under these hypotheses, the average $v$-adic valuation of the preperiodic points is $0$. However, the average minimum $v$-adic valuation of a pair of preperiodic points is negative, and can be bounded uniformly from above in terms of $d$ and $g_v$. Lemma \ref{lem:heightwin} gives an upper bound on this average minimum for pairs of preperiodic points, one that is uniform across monic polynomials of a fixed degree having $0$ as a preperiodic point.

\begin{lem}\label{lem:heightwin}Let $f(z)\in K[z]$ be monic of degree $d\ge2$ with $0$ a preperiodic point of $f$. Let $\epsilon>0$, and let $T\subseteq K$ be a finite set. There are constants $\tau=\tau(\epsilon,d)>0$, $\eta=\eta(\epsilon,d)>0$ and $N=N(\epsilon,d)$ such that if $|T|\ge N$ and \[\fr{1}{|T|}\sum_{P_i\in T}\hat{h}_f(P_i)\le\tau h_{\textup{crit}}(f),\] then at least $\eta|T|^2$ elements $(P_i,P_j)\in T^2$ satisfy \begin{equation}\label{eqn:heightwin}\sum_{v\in M_K^0\setminus\mathscr{S}_d}r_v\log\max\{|P_i|_v,|P_j|_v\}\ge\left(\fr{1}{d^2}-\epsilon\right) h_{\textup{crit}}(f).\end{equation}\end{lem}

\begin{proof}Let $0<\delta<1$, and let $\tau$ be such that \[\frac{1}{|T|}\sum_{P_i\in T}\hat{h}_f(P_i)\le\tau h_{\textup{crit}}(f).\] Let $\epsilon_1>0$. It follows from Theorem \ref{thm:globaleq} and \cite[Proposition 3.7]{Looper:UBCpolys} that if $\tau\ll_{d,\epsilon_1,K}1$ and $|T|\gg_{d,\epsilon_1,K}1$, then \begin{equation}\label{eqn:triviallowerbound}\frac{1}{|T|}\sum_{P_i\in T}\sum_{v\in M_K\setminus\mathscr{S}_d}r_v\log|P_i|_v\ge-\sum_{v\in M_K\setminus\mathscr{S}_d}r_v\epsilon_1g_v.\end{equation}
		
Let $\epsilon_2>0$. Theorem \ref{thm:globaleq} implies that if $|T|\gg_{d,\delta,\epsilon_2,K}1$ and $\tau\ll_{d,\delta,\epsilon_2,K}1$, there is a $\delta$-slice $\Sigma$ of bad places $v\in M_K^0\setminus\mathscr{S}_d$ such that $T$ is $\epsilon_2$-equidistributed at level $1$. For each $v\in\Sigma$, at least $\frac{(1-\epsilon_2)^2|T|^2}{d^2}$ elements $(P_i,P_j)\in T^2$ have the property that $|P_i|_v=e^{g_v}$ and $|P_j|_v\le 1$. (Here it is key that $0\in\mathcal{K}_v$. The upper bound on $|P_j|_v$ comes from the fact that $\delta_v(\mathcal{B}_{1,i},\mathcal{B}_{1,i})\le 1$ for each disk component $\mathcal{B}_{1,i}$ of $\mathcal{E}_1$, which can be seen in the proof of \cite[Proposition 3.4]{Looper:mincanht}.) For such $v$ and such $P_i,P_j$, \[\log\max\{|P_i|_v,|P_j|_v\}-\log|P_j|_v\ge g_v.\] Transposing local averages to global ones, it follows that \begin{equation}\label{eqn:lowerbound}\frac{1}{|T|^2}\sum_{v\in M_K^0\setminus\mathscr{S}_d}\sum_{(P_i,P_j)\in T^2}r_v\log\max\{|P_i|_v,|P_j|_v\}-r_v\log|P_j|_v\ge\delta\sum_{v\in M_K^0\setminus\mathscr{S}_d}r_v\frac{(1-\epsilon_2)^2}{d^2}g_v.\end{equation} Combining (\ref{eqn:lowerbound}) with (\ref{eqn:triviallowerbound}) yields \begin{equation}\label{eqn:avglb}\frac{1}{|T|^2}\sum_{v\in M_K^0\setminus\mathscr{S}_d}\sum_{(P_i,P_j)\in T^2}r_v\log\max\{|P_i|_v,|P_j|_v\}\ge\sum_{v\in M_K^0\setminus\mathscr{S}_d}r_v\left(\frac{(1-\epsilon_2)^2}{d^2}-\epsilon_1\right)g_v\end{equation} whenever $|T|\gg_{d,\delta,\epsilon,\epsilon_1,K}1$ and $\tau\ll_{d,\delta,\epsilon,\epsilon_1,K}1$.
	
Suppose the number of elements $(P_i,P_j)\in T^2$ satisfying (\ref{eqn:heightwin}) is $\eta|T|^2$, where $\eta\in[0,1]$. We have \[\frac{1}{|T|^2}\sum_{(P_i,P_j)\in T^2}\sum_{v\in M_K^0\setminus\mathscr{S}_d}\log\max\{|P_i|_v,|P_j|_v\}\le h_{\textup{crit}}(f)\] so long as $|T|$ is sufficiently large and $\tau$ is sufficiently small, relative to the above constants, so under this assumption we get that \begin{equation*}\begin{split}\frac{1}{|T|^2}\sum_{(P_i,P_j)\in T^2}\sum_{v\in M_K^0\setminus\mathscr{S}_d}r_v\log\max\{|P_i|_v,|P_j|_v\}&\le \left((1-\eta)\left(\fr{1}{d^2}-\epsilon\right)+\eta\right)h_{\textup{crit}}(f)\\&\le\frac{1}{\delta}\sum_{v\in M_K^0\setminus\mathscr{S}_d}\left((1-\eta)\left(\fr{1}{d^2}-\epsilon\right)+\eta\right)g_v.\end{split}\end{equation*} If $\eta,\epsilon_1$ and $\epsilon_2$ are sufficiently small and $\delta$ is sufficiently large, relative to $\epsilon$, the right-hand side is strictly smaller than the right-hand side of (\ref{eqn:avglb}), a contradiction. That $\delta=\delta(\epsilon_2,|T|)\to1$ as $|T|\to\infty$ follows from \cite[Proposition 4.1]{Looper:UBCpolys}.\end{proof}

\section{Proof of main theorems}

In this section, we prove Theorems \ref{thm:Bogomolov}, \ref{thm:UBCsuperattracting} and \ref{thm:mincanht}. Let $f\in K[z]$ be a monic polynomial of degree $d\ge2$. Let $S_1=\mathscr{S}_d\cup M_K^\infty$, and let $S_2$ be the set of places of good reduction for $f$. 

\begin{definition*} Let $\epsilon>0$. We say $\alpha\in K^*$ is \emph{$\epsilon$-good} if \[\sum_{v\in S_1\cup S_2}r_v\log|\alpha^{-1}|_v\le\epsilon h_{\textup{crit}}(f).\]\end{definition*}

The following proposition says that typical differences of preperiodic points are $\epsilon$-good, provided one considers averages across sets of preperiodic points with cardinalities large relative to $\epsilon$. In our application of $abc$ in the proofs of Theorems \ref{thm:UBCsuperattracting} and \ref{thm:mincanht}, this allows the contributions to the radical coming from places of good reduction to be assumed to be arbitrarily small, relative to the height of the map.

\begin{prop}\cite[Proposition 5.1]{Looper:UBCunicrit}\label{prop:adgoodbars} Let $f(z)\in K[z]$ be a monic polynomial of degree $d\ge2$, and let $\epsilon>0$. Let $T\subseteq K$ be a finite set, and let \[\mathcal{T}=\{P_j-P_i:P_i,P_j\in T\}.\] There are constants $\tau=\tau(d,\epsilon)>0$ and $N=N(d,\epsilon)$ such that if $|T|\ge N$ and \[\fr{1}{|T|}\sum_{P_i\in T}\hat{h}_f(P_i)\le\tau h_{\textup{crit}}(f),\] then at least $(1-\epsilon)|\mathcal{T}|$ elements of $\mathcal{T}$ are $\epsilon$-good.\end{prop}

\begin{proof}[Proof of Theorems \ref{thm:UBCsuperattracting} and \ref{thm:mincanht}] Fix $K$ a number field or a one-dimensional function field of characteristic zero. Let $d\ge3$, let $g(z)\in\mathfrak{S}_{d,1}(K)$, and let $T\subseteq K$ be such that \[\fr{1}{|T|}\sum_{P_i\in T}\hat{h}_g(P_i)\le\tau h_{\textup{crit}}(g).\] Let $\epsilon>0$. If $|T|\gg_{d,\epsilon}1$ and $\tau\ll_{d,\epsilon}1$, then by Proposition \ref{prop:adgoodbars} we can choose $P_1,P_2,P_3\in T$ such that $P_1-P_3,P_2-P_3$, and $P_1-P_2$ are $\epsilon$-good. Let $\mu(z)=z-P_3$, and $\tilde{g}=\mu g\mu^{-1}$. Then $\mu(P_3)=0$, $\mu(P_i)-\mu(P_j)=P_i-P_j$ for $1\le i,j\le3$, and $\tilde{g}\in K[z]$ with $K$-rational preperiodic points being translates of the $K$-rational preperiodic points of $g$. Therefore without loss, we may replace $g$ with $\tilde{g}$, replace $T$ with $T-P_3$, and thereby assume that $P_3=0$. We conjugate $g$ by an appropriate scaling to obtain a monic conjugate $f\in K'[z]$, where $K'/K$ is a finite extension containing the image of $T$ under this conjugation. If $|T|\gg_{d,\epsilon}1$ and $\tau\ll_{d,\epsilon}1$, we may further assume by Lemma \ref{lem:heightwin} that (\ref{eqn:heightwin}) holds for $P_i=P_1$ and $P_j=P_2$. Write $P=(-P_1,P_2,P_1-P_2)\in T^3$. Let $\eta=\eta(\epsilon,d)$ be as in Lemma \ref{lem:heightwin}, let $k\in\ZZ_{>0}$, and let $0<\delta<1$. By Proposition \ref{prop:uniformram} and Theorem \ref{thm:globaleq}, if $|T|\gg_{k,d,\delta}1$ and $\tau\ll_{k,d,\delta}1$, then there is a $\delta$-slice $\Sigma$ of bad places $v\in M_{K'}^0\setminus\mathscr{S}_d$ such that if $v\in\Sigma$, then \[|e^{g_v}|_v\in(|K_v|^\times)^k.\] Hence for $v\in\Sigma$, \begin{equation}\label{eqn:kcut}r_v\lambda_{\textup{crit},v}(f)\ge kN_v.\end{equation} Moreover, for any bad place $v\in M_{K'}^0\setminus\mathscr{S}_d$, we have from the Newton polygon of $f$ that \begin{equation}\label{eqn:diam}N_v\le dr_v\lambda_{\textup{crit},v}(f).\end{equation} (Here we are using the fact that since $f$ is monic, $\lambda_{\textup{crit},v}(f)=g_v$ for all $v\in M_{K'}^0\setminus\mathscr{S}_d$, and $g_v$ can be read off the slope of the rightmost segment of the Newton polygon of $f$ at $v$.) For $P=(-P_1,P_2,P_1-P_2)\in T^3$ as above, we write \[\textup{rad}'(P)=\sum_{v\in I(P)}e_vN_v,\] where $e_v$ is the ramification index of $v\in M_{K'}^0$ over the prime lying below $v$ in $M_K^0$. We note that if $K'$ is taken to be the minimal field of definition for $f$, then $e_v\le d-1$ for all $v\in M_{K'}^0$. Then (\ref{eqn:kcut}), (\ref{eqn:diam}) along with Lemma \ref{lem:heightwin} and \cite[Proposition 4.1]{Looper:UBCpolys} imply that when $|T|\gg_{\epsilon,d,\delta}1$ and $\tau\ll_{\epsilon,d,\delta}1$, \begin{equation}\begin{split}h(P)-\text{rad}'(P)\ge&\left(\fr{1}{d^2}-\epsilon\right)h_{\textup{crit}}(f)-\sum_{v\in I(P)}e_vN_v\\\ge&\left(\fr{1}{d^2}-\epsilon\right)h_{\textup{crit}}(f)-\sum_{\substack{v\in I(P)\\\textup{good}}}e_vN_v-\sum_{\substack{v\in M_K^0\setminus\mathscr{S}_d\\\textup{bad}}}e_vN_v-\sum_{v\in\mathscr{S}_d}e_vN_v\\\ge&\left(\fr{1}{d^2}-\epsilon\right)h_{\textup{crit}}(f)-3(d-1)\epsilon h_{\textup{crit}}(f)\\&-\sum_{\substack{v\in M_K^0\setminus\mathscr{S}_d\\\textup{bad}}}(d-1)r_v\left(\frac{\delta}{k}+d(1-\delta)\right)\lambda_{\textup{crit},v}(f)-(d-1)\epsilon h_{\textup{crit}}(f)\\\ge&\left(\frac{1}{d^2}-5\epsilon(d-1)-(1-\epsilon)(d-1)\left(\frac{\delta}{k}+d(1-\delta)\right)\right)h_{\textup{crit}}(f).\end{split}\end{equation} For $\epsilon\ll1$, $k\gg1$, and $\delta$ sufficiently close to $1$, this is a positive proportion of $h_{\textup{crit}}(f)$ uniformly bounded from below. As $P$ is in fact defined over $K$, with \[\textup{rad}'(P)=\textup{rad}_K(P),\] this violates the $abc$-conjecture when $h_{\textup{crit}}(f)=h_{\textup{crit}}(g)\gg_{\epsilon,d,\delta}1$.
	
Finally, suppose $h_{\textup{crit}}(g)$ lies below this bound. Then $g$ has a bounded number of places of bad reduction. In the case where $K$ is a number field, we have that by \cite{Ingram:critmod}, $h_{\textup{crit}}(g)\asymp h_{\mathcal{M}_d}(g)$, where $h_{\mathcal{M}_d}$ is the height associated to an embedding of the moduli space $\mathcal{M}_d$ of degree $d$ rational functions into projective space. Northcott's Theorem then completes the proof. If $K$ is a function field, and $g$ has at least one place of bad reduction, \cite[Main Theorem]{Benedetto} completes the proof. Otherwise, $g$ is isotrivial \cite[Theorem 1.9]{Baker:Northcott}, and \cite[Lemma 6.2]{Looper:UBCpolys} finishes the proof. This proves both theorems in the case $m=1$. The case $m\ge 2$ proceeds from passing to an appropriate iterate of the map to reduce to the case where $m=1$.\end{proof}

\begin{proof}[Proof of Theorem \ref{thm:Bogomolov}] Suppose for a contradiction that $\{P_i\}_{i=1}^\infty$ is a sequence of points with $\hat{h}_f(P_i)\to0$ as $i\to\infty$, and $K(\{P_i\}_{i=1}^\infty)/K$ an extension with abelian Galois closure. Moreover, we may suppose without loss that $f$ is monic; indeed, we can conjugate $f$ over a finite extension of $K$ to obtain a monic conjugate, and without loss replace $K$ by this extension. As before, write $L^i$ for the Galois closure of $K(P_i)$. Let $v$ be a place of bad reduction for $f$. Since the Galois closure of $K(\{P_i\}_{i=1}^\infty)/K$ is abelian, for any $P_i$, we have the equality of decomposition groups $\textup{Decomp}(\mathfrak{Q}_i|v)=\textup{Decomp}(\mathfrak{Q}_j|v)$ for any prime $\mathfrak{Q}_j$ of $L^i$ lying above $v$. 
	
Let $\mathcal{B}$ be the largest Berkovich disk about $p_0$ contained in $\mathcal{K}_v$. Let $\mathcal{B}'$ be a Berkovich disk properly containing $\mathcal{B}$ such that $\log\gamma_v(\mathcal{B}')\le\log\gamma_v(\mathcal{K}_v)-\eta$, where $\eta>0$. Fix an embedding of $\bar{K}$ into $\bar{K}_v\subseteq\mathbf{A}_v^1$, let $\epsilon=\epsilon(f,\mathcal{B}')$ be as in Corollary \ref{cor:basinnbhd}, and let \[T_i=\{\sigma(P_i):\sigma\in\textup{Decomp}(\mathfrak{Q}_i\mid v)\},\] where $\mathfrak{Q}_i$ is as in Corollary \ref{cor:shape}. From Corollary \ref{cor:shape}, we know that there is a subsequence $\{P_{ij}\}$ such that $|T_{i_j}|\to\infty$ as $i_j\to\infty$. To ease the notation, replace $\{P_i\}$ with this subsequence. For any $i$, let $S_i$ be the set of primes $\mathfrak{Q}_j$ of $L^i$ above $v$ such that $P_i\in\mathcal{B}'$. When $P_i\in\mathcal{B}'$, we have $T_i\subseteq\mathcal{B}'$, as \[\textup{Decomp}(\mathfrak{Q}_i\mid v)=\textup{Decomp}(w\mid v)\hspace{5mm}\text{ for all }w\mid v\] acts by isometries on $L_w^i$ for all $w\mid v$. Finally, we note that $\log\gamma_v(\mathcal{B}')<0$, as $\log\gamma_v(\mathcal{K}_v)=0$ because $f$ is monic. It thus follows from Corollary \ref{cor:basinnbhd} and our choice of $\epsilon$ that given $\epsilon'>0$ and $i\gg_{\epsilon'}1$, \begin{equation}\label{ineq:prodform}\begin{split} 0&=\sum_{w\in M_{L^i}}r_w\log d_w(T_i)-r_w\log\gamma_w(\mathcal{K}_w) \\&=\sum_{w\in S_i}r_w\log d_w(T_i)-r_w\log\gamma_w(\mathcal{K}_w)+\sum_{w\in M_{L^i}\setminus S_i}r_w\log d_w(T_i)-r_w\log\gamma_w(\mathcal{K}_w) \\&\le\sum_{w\in S_i}-\eta r_w\lambda_{\textup{crit},w}(f)+\sum_{\substack{w\mid v\\w\notin S_i}}\epsilon'r_w\lambda_{\textup{crit},w}(f)+\sum_{\substack{w\in M_{L^i}^0\\w\nmid v}}\epsilon'r_w\lambda_{\textup{crit},w}(f)+\sum_{w\in M_{L^i}^\infty}\epsilon'\\&\le(-\eta\epsilon+\epsilon'(1-\epsilon))r_v\lambda_{\textup{crit},v}(f)+\sum_{\substack{v'\in M_K^0\\v'\ne v}}\epsilon'r_v\lambda_{\textup{crit},v'}(f)+\sum_{w\in M_{L^i}^\infty}\epsilon'. \end{split}\end{equation} If $\epsilon'$ is sufficiently small compared to $\eta$ and $\epsilon$, the right-hand side of (\ref{ineq:prodform}) is strictly less than $0$ yielding a contradiction.\end{proof}


\begin{thebibliography}{1}
	
	\bibitem{Baker:Northcott} M.~Baker. A finiteness theorem for canonical heights attached to rational maps over function fields. \emph{J. Reine Angew. Math }\textbf{626}(2009) 205--233. MR 2492995.
	
	
	\bibitem{Baker2} M.~Baker. Lower bounds for the canonical height on elliptic curves over abelian extensions. \emph{International Math. Research Notices} 29 (2003), 1571--1589. MR 1979685.
	
	\bibitem{Baker} M.~Baker, S.~Payne and J.~Rabinoff. On the structure of nonarchimedean analytic curves. \emph{Contemporary Mathematics} \textbf{605}(2013) 93--121.
	
	\bibitem{BakerPetsche} M.~Baker and C.~Petsche. Global discrepancy and small points on elliptic curves. \emph{International Math. Research Notices} 61 (2005), 3791--3834. MR 2205235.
	
	\bibitem{BakerRumely} M.~Baker and R.~Rumely. \emph{Potential Theory and Dynamics on the Berkovich Projective Line}. Mathematical Surveys and Monographs, vol. 159. AMS, Providence, 2010.
	
	\bibitem{BakerSilverman} M.~Baker and J.~H.~Silverman. A lower bound for the canonical height on abelian varieties over abelian extensions. \emph{Mathematical Research Letters} 11 (2004) 377--396. MR 2067482.
	
	\bibitem{Benedetto} R.~Benedetto. Preperiodic points of polynomials over global fields. \emph{J.~Reine Angew.~Math.} \textbf{608}(2007) 123--153. MR 2339471.
	
	\bibitem{DeMarcoRumely} L.~DeMarco and R.~Rumely. Transfinite diameter and the resultant. \emph{J.~Reine Angew.~Math.} \textbf{611}(2007) 145--161. MR 2361087.
	
	\bibitem{DvornicichZannier} R.~Dvornicich and U.~Zannier. Cyclotomic Diophantine problems (Hilbert irreducibility and invariant sets for polynomial maps). \emph{Duke Math.~J.} Volume 139, Number 3 (2007), 527--554. MR 2350852.
	
	\bibitem{Ingram:critmod} P.~Ingram. The critical height is a moduli height. \emph{Duke Math J. }\textbf{167}(2018) 1311--1346. MR 3799700.
	
	\bibitem{Ingram:PCF} P.~Ingram. A finiteness result for post-critically finite polynomials.
	\emph{International Mathematics Research Notices }\textbf{3}(2012) 524--543. MR 2885981.
	
	\bibitem{Looper:UBCunicrit} N.~R.~Looper. Dynamical uniform boundedness and the $abc$-conjecture. \emph{Invent.~Math. }\textbf{225}(1) 1--44.
	
	\bibitem{Looper:mincanht} N.~R.~Looper. A lower bound on the canonical height for polynomials. \emph{Mathematische Annalen} 373(3), 1057--1074. MR 3953120.
	
	\bibitem{Looper:UBCpolys} N.~R.~Looper. The Uniform Boundedness and Dynamical Lang Conjectures for polynomials. Preprint available at arXiv:2105.05240.
	
	\bibitem{Mason} R.~C.~Mason. Diophantine Equations over Function Fields. London Mathematical Society Lecture Note Series 96. Cambridge University Press, New York, 1984.
	
	\bibitem{MortonSilverman} P.~Morton and J.~H.~Silverman. Rational periodic points of rational functions. \emph{ International Mathematics Research Notices }\textbf{2}(1994) 97--110. MR 1264933.
	
	\bibitem{Pottmeyer} L.~Pottmeyer. Heights and totally $p$-adic numbers. \emph{Acta Arithmetica }\textbf{171} (2015), 277--291. MR 3416948.
	
	\bibitem{Ruppert} W.~Ruppert. Torsion points of abelian varieties in abelian extensions. Preprint available at arXiv:9803169.
	
	\bibitem{Silverman3}J.~H.~Silverman. \emph{The Arithmetic of Dynamical Systems}, volume 241 of Graduate Texts in Mathematics. Springer, New York, 2007.
	
	\bibitem{Silverman:abc} J.~H.~Silverman. The S-unit equation over function fields. \emph{Proc. Camb. Philos. Soc. }\textbf{95} (1984) 3--4.
	
	\bibitem{Zarhin} Y.~G.~Zarhin. Endomorphisms and torsion of abelian varieties. \emph{Duke Math. J.} 54 (1987), 131--145. MR 0885780.
	
	
\end{thebibliography}
\end{document}